\documentclass[12pt,a4paper]{amsart}

\usepackage[utf8]{inputenc}
\usepackage[T1]{fontenc}
\usepackage{lmodern}
\usepackage{standard}
\usepackage{amsrefs}
\usepackage{enumerate}
\usepackage[usenames,dvipsnames]{color}
\usepackage[colorlinks=true,linkcolor=Red,citecolor=Green]{hyperref}

\numberwithin{theorem}{section}

\DeclareMathOperator{\singsupp}{sing-supp}
\newcommand{\rd}{F}
\newcommand{\trd}{\widetilde{F}}
\DeclareMathOperator{\Tr}{Tr}
\DeclareMathOperator{\F}{\mathcal{F}}
\newcommand{\Sph}{\mathbb{S}}
\newcommand{\symb}{\Gamma}
\newcommand{\symbcl}{\Gamma_{\mathrm{cl}}}
\newcommand{\calc}{G}
\newcommand{\calccl}{G_{\mathrm{cl}}}
\newcommand{\sob}{H_{\mathrm{iso}}}
\newcommand{\WFiso}{\WF_{\mathrm{iso}}}
\newcommand{\hamvf}{\mathsf{H}}
\newcommand{\schwartz}{\mathcal{S}}
\newcommand{\pbis}{\mathsf{p}}
\newcommand{\cl}{\mathrm{cl}}

\newcommand{\Xray}{\mathsf{X}}

\newcommand{\tprop}{\widetilde{U}}

\newcounter{StepCounter}[theorem]
\newcommand{\step}[1]{\stepcounter{StepCounter} {\bf Step \arabic{StepCounter}: #1: }}

\title[Refined Weyl law]{Refined Weyl law for homogeneous perturbations of the
  harmonic oscillator}

\author[M. Doll]{Moritz Doll}
\address{Institut für Analysis, Leibniz Universität Hannover, Welfengarten 1, \newline\indent D-30167 Hannover, Germany}
\email{doll[AT]math.uni-hannover.de}

\author[O. Gannot]{Oran Gannot}
\address{Department of Mathematics, Northwestern University, 2033 Sheridan Rd., \newline\indent Evanston IL 60208, USA}
\email{gannot[AT]northwestern.edu}

\author[J. Wunsch]{Jared Wunsch}
\address{Department of Mathematics, Northwestern University, 2033 Sheridan Rd., \newline\indent Evanston IL 60208, USA}
\email{jwunsch[AT]math.northwestern.edu}

\thanks{
The first author was supported by the DFG GRK-1463 and would like to thank Northwestern University for its hospitality.
The second and third authors gratefully acknowledge the support of NSF grants DMS-1502632 and DMS-1600023, respectively.}

\begin{document}

\begin{abstract}
Let $H$ denote the harmonic oscillator Hamiltonian on $\mathbb{R}^d,$
perturbed by an isotropic pseudodifferential operator of order $1.$
We consider the Schr\"odinger propagator $U(t)=e^{-itH},$ and find
that while $\singsupp \Tr U(t) \subset 2 \pi \mathbb{Z}$ as in the unperturbed case, there exists a large class of perturbations in
dimension $d \geq 2$ for which the singularities of $\Tr U(t)$ at nonzero
multiples of $2 \pi$ are weaker than the singularity at $t=0$.  
The remainder term in the Weyl law is of order $o(\lambda^{d-1})$, improving in these cases the $O(\lambda^{d-1})$ remainder previously established by Helffer--Robert.
\end{abstract}

\maketitle

\section{Introduction}

\subsection{Main results} \label{section:mainresults}
Let $H_0=\frac 12 (\Lap+\smallabs{x}^2)$ denote the isotropic harmonic
oscillator on $\RR^d$, where $\Lap$ is the non-negative Laplacian. Thus $H_0$ is the Weyl quantization $H_0 = \Op_W(p_2)$, where $p_2 = (1/2)(|x|^2 + |\xi|^2)$. Consider a perturbation
\[
 H=\Op_W(p),
\]
where $p$ differs from $p_2$ by a classical \emph{isotropic $1$-symbol}. In other words, $p$ admits an asymptotic expansion
\begin{equation} \label{eq:pasymptotic}
    p \sim p_2 + p_1 +p_0+\dots,
\end{equation}
where each $p_j$ is homogeneous of degree $j$ jointly in $(x,\xi)$. Furthermore, assume that $p$ is real valued, hence $H^*=H$ by properties of the Weyl calculus.

Since $p_2(x,\xi) > 0$ for $(x,\xi) \not = 0$,
the resolvent of $H$ is compact and $H$ has discrete spectrum
\begin{align*}
   \lambda_1 \leq \lambda_2 \leq \dots \to +\infty,
\end{align*}
where each eigenvalue is listed with multiplicity.
Let $E_\lambda$ denote the corresponding spectral projector onto $(-\infty,\lambda]$, so if $N(\lambda) = \sum_{\lambda_j \leq \lambda} 1$ is the counting function, then $N(\lambda) = \Tr E_\lambda$. Moreover, 
the Fourier transform of the spectral measure satisfies $U(t) = \F_{\lambda \to t} dE_\lambda$, where $U(t)$ is the propagator for the time-dependent Schrödinger equation
\begin{equation*}
\left\{\begin{aligned}
(i\pa_t - H)U(t) &= 0\\
U(0) &= \id.
\end{aligned}\right.
\end{equation*}
This implies that
\begin{equation} \label{eq:Ntrace}
    \F_{\lambda \to t} N'(\lambda) = \Tr U(t),
\end{equation}
where the trace of $U(t)$ is defined as a tempered distribution (cf. \cite{DuGu75} and \S \ref{subsect:traces}). It is clear from \eqref{eq:Ntrace} that there is a relationship between the singularities of $\Tr U(t)$ and the growth of $N(\lambda)$ as $\lambda \rightarrow \infty$. A proof of the following Poisson relation can be found in
\cite{HeRo82}, but we will give a short and simple proof in the
special case of interest here:

\begin{proposition}\label{prop:wavetrace}
	Singularities of the Schr\"odinger trace $\Tr U(t)$ satisfy 
	\[
	\singsupp \Tr U(t) \subset 2 \pi \ZZ.
	\]
\end{proposition}

Let $\hamvf_0$ denote the Hamilton vector field of $p_2=(1/2)
(|x|^2+|\xi|^2)$, whose flow $(x(t),\xi(t)) = \exp(t \hamvf_0)(x_0,\xi_0)$ satisfies
\begin{align*}
	x(t) &= \cos(t)x_0 + \sin(t)\xi_0,\\
	\xi(t) &= \cos(t)\xi_0 - \sin(t)x_0.
\end{align*}
Given a function $f \in \CI(\RR^{2d-1})$, let $\Xray f$ denote\footnote{This is a kind of X-ray transform, hence the notation.} the average of $f$ over one period of the flow,
\begin{equation} \label{eq:Xray}
\Xray f(x,\xi)=\int_0^{2\pi} f(\exp(t\hamvf_0)(x,\xi))\, dt.
\end{equation}
When restricted to the sphere, $\Xray f$ can also be viewed as the average of $f$ over the fibers of the complex Hopf fibration $\Sph^{2d-1} \to \mathbb{CP}^{d-1}$. Indeed, consider the map
\[
(x,\xi) \mapsto x+ i\xi,
\]
which identifies $\RR^{2d}$ with $\CC^d$. This map intertwines the action of $\exp(t\hamvf_0)$ with complex rotations $z \mapsto e^{-it} z$, and by restriction to $\Sph^{2d-1}$ the latter action induces the complex Hopf fibration $\Sph^{2d-1} \to \mathbb{CP}^{d-1}$ with fiber $\Sph^1$.

The following theorem, which constitutes the main result of this
paper, shows that the singularities of $\Tr U(t)$ at
nonzero times, and hence also the remainder term in the Weyl law,
depend on properties of $\Xray p_1$  (recall from
\eqref{eq:pasymptotic} that $p_1$ is the subprincipal symbol of $H$).

\begin{theorem}\label{maintheorem}
Assume that, when restricted to $\Sph^{2d-1}$, the set where
$\nabla \Xray p_1$ vanishes to infinite order has measure zero.
If $\chi\in \mathcal{C}_c^\infty((-2\pi,2\pi))$, then for
all $n \in \ZZ \backslash \{0\},$
\begin{equation} \label{eq:infiniteorder}
    \F^{-1}_{t\to \lambda} \chi(t-2\pi n) \Tr U(t)=o(\lambda^{d-1}).
\end{equation}
If $\Xray p_1$ is Morse--Bott on $\Sph^{2d-1}$ with
$k>0$ nondegenerate directions, then 
\begin{equation} \label{eq:morsebott}
    \F^{-1}_{t\to \lambda} \chi(t-2\pi n) \Tr
    U(t)=O(\lambda^{d-1-k/4}).
\end{equation}
In either of the cases considered above, there holds the Weyl formula
    \begin{equation} \label{eq:twotermweyl}
        N(\lambda) = (2\pi)^{-d}\int_{\{p_2 + p_1 \leq \lambda\}} dxd\eta  - (2\pi)^{-d}\int_{\{p_2 = \lambda\} } p_0(x,\eta) \frac{dS}{|\nabla p_2|} + o(\lambda^{d-1}).
    \end{equation}
\end{theorem}

Observe that $\Xray p_1$ is never Morse since it is constant along the integral curves of $\hamvf_0$. On the other hand, the pullback of a Morse function on $\mathbb{CP}^{d-1}$ by the complex Hopf fibration yields a function $p_1$ on $\Sph^{2d-1}$ such that $\Xray p_1$ admits $2d-2$ nondegenerate directions. Thus in any dimension $d\geq2$ there are always examples of $p_1$ satisfying the Morse--Bott hypothesis of Theorem \ref{maintheorem}.

The two-term Weyl asymptotic  \eqref{eq:twotermweyl} in Theorem \ref{maintheorem} should be viewed as a refinement of the asymptotic formula
\begin{equation} \label{eq:HRasymptotic}
N(\lambda) = (2\pi)^{-d}\lambda^d \int_{\{p_2 \leq 1\}} dxd\eta - (2\pi)^{-d}\lambda^{d-1/2} \int_{\{p_2 = 1\}} p_1 \frac{dS}{|\nabla p_2|} + O(\lambda^{d-1})
\end{equation} 
established earlier by Helffer--Robert \cite{HeRo81}. Indeed, \eqref{eq:HRasymptotic} is recovered from the leading order term in \eqref{eq:twotermweyl} by writing the volume of $\{p_2 + p_1 \leq \lambda\}$ as $\lambda^d$ times the volume of  $\{p_2 + \lambda^{-1/2} p_1 \leq 1\}$ and expanding the latter volume in powers of $\lambda^{-1/2}$.

The necessity of a nondegeneracy hypothesis on
$p_1$ in Theorem~\ref{maintheorem} is apparent already from the unperturbed harmonic oscillator $H_0$. Its eigenfunctions are
given by products of Hermite functions, defined for a multiindex
$\alpha=(\alpha_1,\dots,\alpha_d)$ by
$$
\psi_{\alpha}(x) = \pi^{-d/4} (2^{\smallabs{\alpha}} \alpha!)^{-1/2} H_\alpha(x)e^{-\smallabs{x}^2/2},
$$
with $H_j$ the $j$'th Hermite polynomial and $H_\alpha=\prod_{j=1}^d H_{\alpha_j}(x_j)$; the corresponding
eigenvalues are
$$
\smallabs{\alpha}+\frac d2.
$$
Thus the eigenvalues are $\lambda=j+d/2$ for $j \in \NN,$ arising with
multiplicity
$$
p(\lambda-d/2,d),
$$
where $p(j,d)$ denotes the the number of ways of writing $j$ as a sum of $d$
nonnegative integers.  Since in fact
$$
p(j,d) =
\begin{pmatrix}
  d+j-1\\j
\end{pmatrix},
$$
and this quantity is bounded below for $j \in \NN$ by a multiple of $j^{d-1},$ the
remainder term in the Weyl law for $H_0$ certainly cannot be
$o(\lambda^{d-1}).$ 

The improvement in the Weyl law is not directly related to the
propagation of singularities: If $u \in \mathcal{S}'$, we show that
\begin{align*}
    \WF(U(2\pi k)u) = \{ (x + k\pa_\xi (\Xray p_1)(0,\xi), \xi): (x,\xi) \in \WF(u)\}.
\end{align*}
If we consider the operator $H = H_0 + \sqrt{H_0},$ for which the
symbol of the perturbation is $p_1(x,\xi) = \sqrt{p_2(x,\xi)}$, we see that singularities at time $t=2\pi k$ are shifted by $2\pi k\pa_\xi |\xi|$.
On the other hand there is no improvement in the Weyl law, because the
eigenvalues of $H$ are $j + d/2 + \sqrt{j + d/2}$
and the multiplicity remains $p(j,d)$.

\subsection{Strategy of proof} \label{section:strategy}
As in \S\ref{section:mainresults}, denote the free Hamiltonian  (namely the exact harmonic oscillator) by $H_0 = \Op_W(p_2)$ and the perturbed one by $H = \Op_W(p)$.
Further, let 
\[
U(t) = e^{-itH}, \quad U_0(t) = e^{-itH_0}, \quad \rd(t) = U_0(-t) U(t)
\] 
be the perturbed, free, and ``reduced''
propagator, respectively. Then, $F(t)$ satisfies the evolution equation
 \begin{equation} \label{Fprime}
\left\{\begin{aligned}
(i\pa_t - P(t) ) \rd(t) &= 0\\
\rd(0) &= \id,
\end{aligned}\right.
\end{equation}
where $P(t) = U_0(-t) (H-H_0) U_0(t)$. The main strategy is
to show, following the methods of Helffer--Robert in \cite{HeRo81}, that $\rd(t)$ has an
oscillatory integral parametrix with an explicit phase function. It is then possible to construct a parametrix for $U(t)$ by composing the
parametrix for $\rd(t)$ with the free propagator $U_0(t)$, whose Schwartz kernel is given explicitly by Mehler's formula. Finally, via another more delicate stationary phase
computation, we arrive at estimates on the singularities of $\Tr
U(t).$  The results on spectral asymptotics then follow via a
known Tauberian theorem.

\subsection{Prior results}
It has been known since the work of Zelditch \cite{Zelditch83} (see
also \cite{Weinstein85}) that
singularities of the propagator for perturbations of the harmonic oscillator by a
symbolic potential $V(x) \in S^0(\RR^d)$ reconstruct at times $t \in  \pi
\ZZ.$ Moreover, if the potential is merely bounded with all its derivatives, Zelditch showed that $\singsupp \Tr U(t) \subset 2 \pi \ZZ$. It was later shown by Kapitanski--Rodnianski--Yajima \cite{KapRodYaj97} that the singular support of $\Tr U(t)$ is contained in $2 \pi \ZZ$ supposing only that the perturbation is
subquadratic.  

More general propagation of singularities for geometric generalizations of the
harmonic oscillator to manifolds with large conic ends (``scattering
manifolds'') was also studied by the third author in \cite{Wunsch99}
and refined by Mao--Nakamura \cite{MaoNak09}, which allows for perturbations in the symbol class
$S^{1-\ep}(\RR^d)$ for any $\ep>0.$  

That something dramatic happens for
potential perturbations in $S^1(\RR^d),$ by contrast, is clear from the results of Doi
\cite{Doi03}, where the author shows that the location in space of the
singularities of the Schr\"odinger propagator at times $t \in \pi \ZZ$
is indeed subject to an interesting geometric shift from this type of
perturbation.

Helffer--Robert \cite{HeRo81} studied the singularity at $t=0$ of the
Schr\"odinger trace (and, consequently, the Weyl law) for the class of perturbations under consideration
here, viz., those that are isotropic operators of order $1.$  While this
class does not include potential perturbations of order $1,$ hence is
perhaps less natural on physical grounds, it is more natural from
the point of view of symplectic geometry.  The analysis in
\cite{HeRo81} was limited to the study of the main singularity at $t=0$, hence did not include the considerations of the global
flow studied here.  The parametrix construction of \cite{HeRo81} is
essential in our work, however, as we extend (a version of) it to long times via
composition with the free propagator.

The novelty of our result lies in the delicate perturbation resulting
from a one-symbol.  This is unlike the case famously considered by Duistermaat--Guillemin in 
\cite{DuGu75} under which a genericity hypothesis on the geodesic flow
yields an improvement to the Weyl law remainder for the Laplacian on a
compact manifold. Here, the most naive version of propagation of
singularities, as described by isotropic wavefront set, is unaffected
by the perturbation.  The perturbative effect can be seen
heuristically as a higher-order
correction to the motion of Lagrangian subspaces of $T^*\RR^n:$ at
times $t\in 2 \pi \ZZ,$ the Lagrangian $N^*\{0\},$ for instance, has evolved
under the bicharacteristic flow to another Lagrangian that is
asymptotic to $N^*\{0\}$ as $\smallabs{\xi}\to \infty,$ but it is the
next-order term in the asymptotics of this Lagrangian that governs the
contribution to the Schr\"odinger trace, and hence to the Weyl law
remainder term.

\section{The isotropic calculus}\label{section:calculus}

We now discuss the calculus of \emph{isotropic} pseudodifferential
operators on $\RR^d,$ sometimes referred to as the \emph{Shubin
calculus} and employed by many authors \cites{Helffer84,Shubin78,Hormander91}. The notation used here follows that of \cite{Shubin78}*{Chapter IV}.
Throughout, $\mathcal{S}=\mathcal{S}(\RR^d)$ will denote the space of
Schwartz functions on Euclidean space.

Define the isotropic symbol class of order $k$ by
\[
    \symb^k = \big\{a(x,\xi)\in \CI(\RR^{d}\times \RR^d)\colon \smallabs{\pa_x^\alpha
    \pa_\xi^\beta a} \leq C_{\alpha\beta}
    \ang{(x,\xi)}^{k-(\smallabs{\alpha}+\smallabs{\beta})}\big\}.
\]
The best constants in each of these bounds define a family of seminorms for which $\symb^k$ is a Frech\'et space. Within this class are distinguished the \emph{classical symbols}
$\symb^k_{\cl}$, namely those enjoying asymptotic expansions
\[
    a \sim a_k+a_{k-1}+\dots,
\]
where $a_{j}$ homogeneous in $(x,\xi)$ of degree $j.$ This of course differs from the usual Kohn--Nirenberg symbol class \begin{equation} \label{eq:KN}
S^k(\RR^m;\RR^n) = \big\{a(y,\eta) \in \CI(\RR^{m} \times \RR^n)\colon \smallabs{\pa_y^\alpha
	\pa_\eta^\beta a} \leq C_{\alpha\beta}
\ang{\eta}^{k-\smallabs{\beta}}\big\},
\end{equation}
which will make a brief appearance in \S\ref{sec:sing-zero}.

The class of isotropic pseudodifferential operators of order $k$ will be denoted
\[
    \calc^k=\{\Op_\bullet (a)\colon a \in \symb^k\},
\]
and $\calccl^k$ the quantizations of $\symbcl^k.$ Here, $\bullet = L,R$ or $W$
are the left, right, or Weyl quantization. The Frech\'et topology on $\calc^k$ is inherited from that of $\symb^k$. Finally, set $\calc = \bigcup_k \calc^k$ and $\calccl = \bigcup_k \calccl^k$, as well as $\calc^{-\infty} = \bigcap_k \calc^k$.

The calculus enjoys the following properties, for the proofs of which the reader is referred to \cite{Helffer84}.  
The main novel feature to bear in mind here is that error terms
in the symbol calculus are consistently \emph{two} orders lower (see
in particular property (\ref{commutatorproperty}) below),
reflecting the improvement in decay of symbols under differentiation
in both space and momentum variables.\footnote{An alternative
notational convention would indeed to take the order of an operator
  to be \emph{half} the order used here, so that, e.g., the harmonic
  oscillator would have order $1$; this would fit better with
  the spectral asymptotics results, but with a cost in confusion about
  orders of growth of symbols.}

\begin{enumerate}[(I)] \itemsep4pt
\item
    $\calc$ is a filtered $*$-algebra, with $\calccl$ as a sub-algebra.
\item
    Differential operators of the form 
    \[
        \sum_{\smallabs{\alpha}+\smallabs{\beta}\leq k}
        a_{\alpha,\beta}x^\alpha D^\beta
    \]
    lie in $\calc^k.$
\item
    There is a principal symbol map\footnote{We could also define the principal symbol as an element in $\symb^k/\symb^{k-2}$, but this has the downside that it is no longer homogeneous for classical symbols.}
    \[
    \sigma_k: \calc^k \to \symb^k/\symb^{k-1}
    \]
    such that the following principal symbol sequence is exact:
    \begin{equation*}
        0 \to \calc^{k-1} \to \calc^k \overset{\sigma_k}{\to} \symb^k/\symb^{k-1} \to 0.
    \end{equation*}
    \item The left, right, and Weyl quantization maps $\Op_\bullet$ all
    map $\symb^m \to \calc^m$, and each satisfies 
    \[
    \sigma_m\circ \Op_\bullet: a \to [a] \in \symb^m/\symb^{m-1}. \footnote{This map is independent of the specific quantization.}
    \]
    The Weyl quantization map further enjoys the exact Egorov
    property\footnote{For more general results on metaplectic
    invariance, cf.\ \cite{Hormander3}*{Theorem 18.5.9}.} for the harmonic oscillator propagator:
    \[
        U_0(-t)\Op_W(a)U_0(t) =\Op_W(a\circ \exp(t\hamvf_0)).
    \]
\item\label{commutatorproperty}
    If $A\in\calc^m,$ $B\in\calc^{m'}$, then $$[A,B]\in\calc^{m+m'-2}$$
    and satisfies
    \begin{equation*}
        \sigma_{m+m'-2}([A,B])
        =\frac{1}{i}\{\sigma_m(A), \sigma_{m'}(B)\},
    \end{equation*}
    with the Poisson bracket indicating the (well-defined) equivalence
    class of the Poisson bracket of representatives of the equivalence
    classes of each of the principal symbols.
\item
    Every $A\in\calc^0$ defines a continuous linear map on $L^2(\RR^d).$
\item
    The isotropic Sobolev spaces, $\sob^s$
    are defined for $s\geq 0$ by
    \[
        f \in \sob^s \Longleftrightarrow Af \in L^2(\RR^d)\ \text{for
        all}\ A \in\calc^s
    \]
    Equivalently, it is enough to require that $Af \in L^2$ for
    a single {\em elliptic} $A\in\calc^s$.
    As usual, ellipticity means that the principal symbol of $A$ in $\symb^s/\symb^{s-1}$ has an
    inverse in $\symb^{-s}/\symb^{-s-1}$.
    This then fixes
    the norm (up to equivalence), 
    \[
    \| u \|_{\sob^s} = \|u\|_{L^2}+\|Au\|_{L^2},
   \]
    on $\sob^s.$
    For $s<0$ the spaces are defined by duality.  For all $m,s\in \RR$ and all
    $A\in\calc^m$
    \[
        A: \sob^s \to\sob^{s-m}
    \]
    is continuous, and moreover the operator norm of $A$ is bounded by seminorm of its total (left, right, or Weyl) symbol. 
\item\label{limitsob}
    The scale of isotropic Sobolev spaces satisfies \[\displaystyle\bigcap_m \sob^m=\mathcal{S}(\RR^d),\quad \bigcup_m
    \sob^m=\mathcal{S'}(\RR^d).\]
    
\item
    An operator in $A \in \calc^m$ is said to be \emph{elliptic} at
    $q\in \Sph^{2d-1}$ if there is an open cone $U$ containing $q$ on which $|(x,\xi)|^{-m}|\sigma_m(A)| \geq c >0$ for $|(x,\xi)|$ sufficiently large; otherwise, it is said to be \emph{characteristic} at $q.$ Let
    $\Sigma_m(A)$ denote the set of characteristic points.
    
\item
    There is an operator wave front set 
    $\WF'$ such that for $A\in\calc$,
    $\WF'(A)$ is a closed conic (in \emph{all} variables) subset of
    $\RR^{2d}$, or equivalently a closed subset of $\Sph^{2d-1};$ it can
    be defined as the essential support of the total symbol, and satisfies:
    \begin{enumerate}[(A)] \itemsep4pt
    \item
        $\WF' A^*=\WF' A,$
    \item
        $\WF'(AB)\subset\WF'(A)\cap\WF'(B),$ 
    \item
        $\WF'(A+B)\subset\WF'(A)\cup\WF'(B),$ 
    \item\label{operatorexists}
        for any $K\subset \Sph^{2d-1}$ closed and $U\subset \Sph^{2d-1}$
        open with
        $K\subset U$, there exists $A\in\calc^0$ such that 
        $\WF'(A)\subset U$ and $\sigma_0(A)=1$ on $K.$
    \item
        if $A\in\calc^m$ is elliptic at $q\in \Sph^{2d-1}$, then
        there exists a microlocal parametrix $G\in\calc^{-m}$ such that
        \begin{equation*}
        q\notin\WF'(GA-\id)\cup \WF'(AG-\id).
        \end{equation*}
    \item
        For each $A\in\calc^k$ the following are equivalent:
        \begin{itemize}
            \item $\WF'(A)=\emptyset,$
            \item $A\in\calc^{-\infty}$,
            \item $A: \schwartz'\to \schwartz.$
        \end{itemize} 
    \end{enumerate}
\item
    The isotropic wavefront set of $u \in \schwartz'$ is defined by
    \[
        \WFiso u = \bigcap_{\substack{A \in \calc^0\\Au \in \schwartz}} \Sigma_0(A),
    \]
    and there is also a scale of wavefront sets relative to the isotropic
    Sobolev spaces defined by
    \[
        \WFiso u = \bigcap_{\substack{A \in \calc^m\\Au \in L^2}} \Sigma_m(A).
    \]
    These sets satisfy the following:
    \begin{enumerate}[(A)] \itemsep4pt 
        \item $\displaystyle \WFiso u =\emptyset$ if and only if $u \in \schwartz.$
        \item $\displaystyle \WFiso^m u =\emptyset$ if and only if $u \in \sob^m.$
        \item $\displaystyle \WFiso Au \subset \WF'A \cap \WFiso u.$
        \item  $\displaystyle \WFiso u =\overline{\bigcup_{m \in \RR} \WFiso^m u}.$
    \end{enumerate}
\end{enumerate}

Since it is somewhat external to the basic features of the calculus,
we also record separately the following result on the relationship of
isotropic and ordinary wavefront set (cf.\ Proposition 2.7 of
\cite{Hormander91}):
\begin{proposition}\label{prop:isoWF}
Let $u \in \schwartz'.$  If $\{(0,\xi)\colon \xi \in \RR^d\} \cap
\WFiso u=\emptyset$ then $u \in \CI.$
\end{proposition}
\begin{proof}
Choose $a_0 \in \symb^0$ such that $a_0 = 1$ in a conic neighborhood of $\{(0,\xi)\colon \xi \in \RR^d\}$
and $\supp a_0 \cap \WFiso(u) = \emptyset$ and set $A_0 = \Op_L(a_0)$.
By the properties of the calculus $A_0 u \in \schwartz$.
and the operator $A_1 = (1 - A_0)$ is given by
\begin{equation*}
    (A_1u)(x) = (2\pi)^{-d} \int e^{i\ang{x,\xi}} (1 - a_0(x,\xi)) \hat{u}(\xi) d\xi.
\end{equation*}
Because $1- a_0$ is supported away from the vertical space $\{(0,\xi) \colon \xi \in \RR^d\}$, we obtain
$|\xi| \leq C (1 + |x|)$ there.
This implies that $\chi(x)A_1u(x)$ is smooth for any cut-off function $\chi \in \CcI$, because $\chi(x)(1-a_0(x,\xi)) \in \CcI$.
\end{proof}
\section{Singularities of the trace}

\subsection{Propagation of isotropic wavefront set}
Since $P(t) = U_0(-t)(H-H_0)U_0(t)$ and $H-H_0 \in \calc^1,$ it follows from the exact Egorov theorem that
\begin{equation}\label{P}
    P(t) \in \calc^1,\quad P(t)^*=P(t),
\end{equation}
and $P(t)$ is in fact a smooth family of such operators. Somewhat surprisingly, the evolution generated
by $P(t)$ does not move around isotropic wavefront set; this uses
essentially the property of the isotropic calculus that errors are two
orders lower.  The analogous result of course fails for usual wavefront set if $P(t)$ is replaced with an ordinary first order, self-adjoint pseudodifferential operator such as $\sqrt{\Lap}.$ 

\begin{lemma}\label{lemma:nonEgorov}
    Let $P(t) \in \calc^1$ be a smooth family of self-adjoint operators, and assume there is a solution $F(t)$ of the equation
    \begin{equation*}
        \left\{\begin{aligned}
            (i\pa_t - P(t) ) \rd(t) &= 0\\
            \rd(0) &= \id
        \end{aligned}\right.
    \end{equation*}
    such that $F \in \mathcal C^0(\RR_t; \mathcal{L}(\sob^s, \sob^s)) \cap \mathcal C^1(\RR_t; \mathcal{L}(\sob^s,\sob^{s-1}))$ for each $s \in \RR$. Then, $\WFiso \rd(t) u=\WFiso u$ for each $u \in \schwartz'$ and $t \in \RR$.
\end{lemma}

\begin{proof}
Suppose that $u \in \schwartz'$, hence there exists $s_0$ such that $u \in \sob^{s_0}$, and by hypothesis $\rd(t)u \in \sob^{s_0}$ for all $t \in \RR$. The goal is to show by induction that for every $k,$ the set $\WFiso^k \rd(t) u$ is invariant; this is trivially true for
$k=s_0$, as the wavefront set remains empty.

Suppose that $U\subset \Sph^{2d-1}$
is open, and $\WFiso^{k} u \cap U=\emptyset.$ The inductive step is completed by showing that
\[
\WFiso^{k-1} \rd(t) u \cap U = \emptyset \Longrightarrow \WFiso^k \rd(t) u \cap U = \emptyset. 
\] 
Let $A \in \calc^k$ be fixed independently of $t$ such that $\WF' A \subset U$. Choose a bounded family 
\[
\{ A_\varepsilon\colon \varepsilon \in [0,1) \} \subset \calc^k
\]
such that $A_0 = A$ and $A_\varepsilon \in \calc^{k-1}$  for each $\varepsilon \in (0,1)$. Furthermore, assume that $\WF' A_\varepsilon \subset U$ for $\varepsilon \in [0,1)$. For instance, let $A_\varepsilon = S_\varepsilon A$, where 
\[
S_\varepsilon = \Op_W ((1+\varepsilon(|x|^2+|\xi|^2))^{-1/2}).
\]
Observe in this case that $A_\varepsilon \rightarrow A$ in the topology of $\calc^{k+1}$. Using \eqref{Fprime}, compute
\begin{equation*}
    \frac{d}{dt} A_\varepsilon\rd(t) =- i P(t) (A_\varepsilon\rd(t))-i [A_\varepsilon,P(t)]\rd(t).
\end{equation*}
Since  $P(t)$ is self-adjoint and $A_\varepsilon \rd(t) u \in L^2(\RR^d)$ by the inductive hypothesis, \begin{equation}\label{notEgorov}
    \begin{aligned}
    \frac{d}{dt}  \smallnorm{A_\varepsilon\rd(t) u}^2
        &= 2\Re\smallang{(d/dt)A_\varepsilon \rd(t)u, A_\varepsilon\rd(t) u}\\
        &=- 2\Re \smallang{i[A_\varepsilon,P(t)] \rd(t) u, A_\varepsilon \rd(t) u}\\
        &\leq 2\smallnorm{A_\varepsilon\rd(t) u}\,  \smallnorm{[A_\varepsilon,P(t)] \rd(t) u}.
    \end{aligned}
\end{equation}
On the other hand, since $A_\varepsilon$ is bounded in $\calc^k$, it
(crucially) follows that $[A_\varepsilon, P(t)]$ is bounded in
$\calc^{k-1}$ for
$\varepsilon \in [0,1)$. Furthermore, the operator wavefront set of
$[A_\varepsilon,P(t)]$ is contained in $U$. Now integrate to find that
\[
 \| A_\varepsilon \rd(t) u \|^2 \leq e^t\|A_\varepsilon u \|^2 + e^t \int_0^t e^{-s} \| [A_\varepsilon,P(s)]F(s)u \|^2 \, ds
\]
for each fixed $t$, where the right hand side is uniformly bounded as
$\varepsilon \rightarrow 0$. From the weak compactness of the unit
ball in $L^2(\RR^d)$, conclude that $A_{\varepsilon_k} \rd(t) u$ has a
weak limit in $L^2(\RR^d)$ along a sequence of $\varepsilon_k
\rightarrow 0$, hence in $\schwartz'(\RR^d)$ as well. On the other
hand,  $A_\varepsilon \rd(t) u \rightarrow A \rd(t) u$ in
$\schwartz'(\RR^d)$, since $A_\varepsilon \rightarrow A$ in
$\calc^{k+1}$. It follows that $A\rd(t)u \in L^2(\RR^d)$, and we have shown
that for $t>0,$
$$
\WFiso F(t) u \subset \WFiso u.
$$
To obtain the reverse inclusion, we repeat the argument above,
integrating a time-reversed version of \eqref{notEgorov} from $t$ to $0$ instead of $0$ to $t.$
\end{proof}

Lemma \ref{lemma:nonEgorov} can be applied directly to the evolution equation \eqref{Fprime}: in that case $F(t) = U_0(-t)U(t)$ and both operators in this composition preserve $\sob^{s}$ for each $s$; thus $F(t)$ has the requisite mapping properties. The invariance of isotropic wavefront set under $U(t)$ follows directly from this lemma:

\begin{proposition}\label{prop:prop}
    For all $u \in \schwartz'$ and $t \in \RR$,
    \[
        \WFiso U(t) u=\WFiso U_0(t) u=\exp(t\hamvf_0) \WFiso u.
    \]
\end{proposition}

\begin{proof}
Since $U_0(t) u=\rd(-t) U(t)u$, the first equality follows from
Lemma~\ref{lemma:nonEgorov}, while the second follows from the exact
Egorov theorem for $U_0(t).$
\end{proof}

Equipped with Proposition \ref{prop:prop}, there is a simple proof of Proposition~\ref{prop:wavetrace} following the strategy of \cite{Wunsch99}. 

\begin{proof}[Proof of Proposition \ref{prop:wavetrace}]
Pick any small interval
$I$ not containing a multiple of $2\pi.$ By compactness of the
sphere, there exists a partition
of unity $\{a_j^2\colon j\in J\}$ of $\Sph^{2d-1}$ such that $a_j \cdot (a_j\circ \exp(t\hamvf_0))=0$ for
all $j\in J$ and $t \in I.$  Using an iterative construction in the calculus, it is possible to find $A_j \in \calc^0$ satisfying $\sigma_0(A_j) = a_j$ and $\WF'(A) \subset \supp a_j$, such that
\[
    \sum A_j^2=\id+R,
\]
where $R \in \calc^{-\infty}$ (cf.\
\cite{Wunsch99}*{Corollary 4.7}). Then, computing in the sense of tempered distributions,
\begin{equation}\label{traceiscyclic}
    \begin{aligned}
        \Tr U(t) &=\Tr  \sum A_j^2 U(t) -R U(t)\\
        &= \Tr  \sum A_j U(t) A_j -R U(t).
    \end{aligned}
\end{equation}
The term $A_j U(t) A_j$ maps $\schwartz'\to \schwartz$ by propagation
of singularities (Proposition~\ref{prop:prop}), as do all its
derivatives, and $R U(t)$ also has this property.  Hence the right hand side of
\eqref{traceiscyclic} and
all its derivatives are bounded for $t \in I,$ so
$\Tr U(t) \in \CI(I).$  This completes the proof of
Proposition~\ref{prop:wavetrace}.
\end{proof} 

\section{Parametrix}

\subsection{Oscillatory integrals} \label{subsect:oscillatory} Throughout the rest of the paper it will be important to consider oscillatory integrals of the form
\begin{equation} \label{eq:oscillatory}
I(a,\psi)(z) = \int e^{i\psi(z,\eta)} a(z,\eta) \,d\eta, \quad (z,\eta) \in \RR^k \times \RR^m,
\end{equation}
where $\psi$ is a real-valued quadratic form in $(z,\eta)$. References for this material are \cite{Helffer84}*{Chapter III} and \cite{AsFu78}.  If $\psi$ satisfies the nondegeneracy hypothesis
\begin{equation} \label{eq:nondegeneratephase}
\mathrm{rank} \begin{pmatrix} \partial^2_{\eta z} \psi  & \partial_{\eta\eta}^2 \psi \end{pmatrix} = k+m,
\end{equation}
then \eqref{eq:oscillatory} defines a distribution $I(a,\psi) \in \schwartz'(\RR^k)$ provided the amplitude $a(z,\eta) \in \CI(\RR^{k+m})$ satisfies
\begin{equation} \label{eq:minamplitude}
|\partial^\alpha_{z,\eta} a(z,\eta)| \leq C_\alpha \left<z\right>^M\left<\eta\right>^M
\end{equation}
for some fixed $M \in \RR$ and every $\alpha$. This also means it is possible to consider phases of the form 
\[
\psi = \psi_2 + \psi_1,
\]
where $\psi_2$ is a quadratic form satisfying \eqref{eq:nondegeneratephase}, and $\psi_1$ is real-valued satisfying the bounds
\[
|\partial^\alpha_{z,\eta} \psi_1(z,\eta)| \leq C_{\alpha} 
\]
for each $|\alpha| \geq 1$. Indeed, for the purposes of regularization, it suffices to absorb $e^{i\psi_1}$ into the amplitude, since $e^{i\psi_1} a$ satisfies \eqref{eq:minamplitude}.

Now suppose that $\psi$ is a real-valued quadratic form in $(x,y,\eta) \in \RR^{d}\times \RR^d \times \RR^m$. If $\psi$ satisfies 
\begin{equation} \label{eq:operatorphase}
\det \begin{pmatrix}
\partial_{xy}^2 \psi & \partial_{x\eta}^2 \psi \\
\partial^2_{\eta y}\psi & \partial^2_{\eta\eta} \psi \end{pmatrix} \neq 0
\end{equation}
and $a(x,y,\eta)$ satisfies \eqref{eq:minamplitude} with $z=(x,y)$, then $I(a,\psi)(x,y)$ is the Schwartz kernel of an operator mapping $\schwartz(\RR^d) \rightarrow \schwartz(\RR^d)$ and $\schwartz'(\RR^d) \rightarrow \schwartz'(\RR^d)$. Furthermore, if $a(x,y,\eta) \in \schwartz(\RR^{2d}\times \RR^m)$, then the corresponding operator is residual, namely it maps $\schwartz'(\RR^d) \rightarrow \schwartz(\RR^d)$.

\subsection{Mehler's formula}
As discussed in \S  \ref{section:strategy}, the goal is to approximate $U(t)$ by first approximating $\rd(t)$ by an operator with oscillatory integral kernel of the form
\begin{align*}
\trd(t)(x,y) = \int e^{i \left< x-y,\eta\right> + i\phi_1(t,x,\eta)} a(t,x,\eta) \,d \eta,
\end{align*}
where $\trd(t) - \rd(t)$ is regularizing in suitable sense, and  $\phi_1$ is an explicit phase function which is homogeneous of degree $1$ in $(x,\eta)$. This is useful since $U(t) = U_0(t)F(t)$, and the Schwartz kernel of $U_0(t)$ is explicitly given by Mehler's formula, which is now recalled.

Begin by defining the phase function
\begin{equation} \label{eq:mehlerphase}
\phi_2(t,x,\eta) = \sec(t) (\left<x,\eta\right> - \sin(t)(|x|^2 + |\eta|^2)/2),
\end{equation}
where $(x,\eta) \in \RR^{2d}$. This is well defined for any $t \notin 2\pi \ZZ \pm \pi/2$, and for any such $t$, the quadratic form $\phi_2(t,x,\eta) - \left<y,\eta\right>$ satisfies \eqref{eq:operatorphase}. It is well known then that the Schwartz kernel of $U_0(t)$ satisfies
\begin{equation*} \label{eq:mehlerformula}
U_0(t)(x,y) = (2\pi)^{-d}\frac{(-1)^n}{\cos(t)^{1/2}} \int e^{i\phi_2(t,x,\eta) - i\left< y,\eta\right>} d\eta,
\end{equation*}
where $n$ is such that $t - 2\pi n \in (-\pi/2,\pi/2)$. Thus $U_0(t)(x,y)$ is of the form \eqref{eq:oscillatory}, where for each fixed $t$ the  amplitude is constant.

\subsection{Parametrix for the reduced propagator}
Recall that the reduced propagator $\rd(t)=U_0(-t)U(t)$ solves the evolution equation
\begin{equation}\label{doi-evolution}
    \left\{\begin{aligned}
        (i\pa_t - P(t) ) \rd(t) &= 0,\\
        \rd(0) &= \id.
    \end{aligned}\right.
\end{equation}
Here $P(t) \in \calccl^1$ is a smooth family of classical isotropic operators, and in the notation of \eqref{eq:pasymptotic} its total Weyl symbol $\pbis(t)$ satisfies
\[
\pbis(t) = (p-p_2) \circ \exp(t \hamvf_0)
\] 
by the exact Egorov theorem. In particular, its homogeneous of degree $1$ principal symbol $\pbis_1(t) = \sigma_1(\pbis(t))$ is simply $\pbis_1(t) = p_1 \circ \exp(t \hamvf_0)$. Define
\begin{equation} \label{eq:phi1}
\phi_1(t,x,\xi) = -\int_0^t p_1\circ \exp(s\hamvf_0)(x,\xi) \, ds,
\end{equation}
noting for future reference that $\phi_1(2\pi n, \bullet ) = -\Xray ^np_1 = -n \Xray p_1$ for each $n \in \ZZ$, where $\Xray p_1$ is given by \eqref{eq:Xray}. 

In the following lemma we construct an oscillatory integral parametrix for $F(t)$. 

\begin{lemma}\label{lem:parametrix-perturbed}
There exists $a \in\CI(\RR_t;  \symbcl^0)$ and an operator $\trd(t)$ with Schwartz kernel 
\begin{equation} \label{eq:approximaterd}
\trd(t)(x,y) = \int e^{i\left<x-y, \xi\right>+i\phi_1(t,x,\xi)}  a (t,x,\xi) \, d\xi
\end{equation}
approximately solving \eqref{doi-evolution} in the sense that
	\begin{equation*}\label{approx-doi-evolution}
	(i\pa_t - P(t) ) \trd(t)\in \CI(\RR_t; \mathcal{L}(\schwartz',\schwartz)), \quad \trd(0) = \id + K,
	\end{equation*}
where $K\colon \schwartz' \rightarrow \schwartz$. Here, the function $\phi_1$ is given by \eqref{eq:phi1}.
\end{lemma}

Note that unlike the construction of \cite{HeRo81} (which we are
adapting to our purposes), this holds for arbitrarily long time.

\begin{proof}
We seek an approximate solution to \eqref{doi-evolution} of the form \eqref{eq:approximaterd}. The starting point is the action of an isotropic pseudodifferential operator on oscillatory integral of the form \eqref{eq:approximaterd}, as in \cite{HeRo81}*{Section III} or \cite{Helffer84}*{Theorem 2.5.1}. In order to apply these results directly, first write $P(t)$ as a left quantization, 
\[
P(t) = \Op_L(\widetilde \pbis(t)),
\] 
where the homogeneous degree $1$ part of $\widetilde\pbis(t)$ is still $\pbis_1(t)$. 

Suppose that $a\in\CI(\RR_t;  \symbcl^0)$ and $\phi_1 \in \CI(\RR_t \times \RR^{2d})$. Let $b(t,x,\xi) = e^{i\phi_1(t,x,\xi)}a(t,x,\xi)$, and then define
\[
c(t,x,\xi) = e^{-i\left<x,\xi\right>}P(t)(e^{i\left< \bullet,\xi\right>}b(t,\bullet,\xi) ).
\] 
Referring to \cite{HeRo81}*{Section III}, it follows that $c$ has an asymptotic expansion
\[
c(t,x,\xi) = \sum_{|\alpha| < N} c_{\alpha}(t,x,y,\xi) + c^{(N)}(t,x,y,\xi),
\]
where $c_\alpha$ is given by the formula
\[
c_{\alpha}(t,x,\xi) = (\alpha!)^{-1} \partial_\xi^{\alpha} \widetilde \pbis(t,x,\xi)D^{\alpha}_x b(t,z,\xi).
\]
Furthermore, given $T >0$ and $t\in [-T,T]$, the remainder $c^{(N)}$ satisfies the uniform bound
\begin{equation} \label{eq:cbound}
|\partial_t^k \partial_x^{\beta} \partial_\xi^{\gamma} c^{(N)}(t,x,\xi)| \leq C_{k\beta\gamma}\left<(x,\xi)\right>^{k+1-N}.
\end{equation}
Disregarding smoothness at $(x,\xi) = 0$ at first, formally apply this result with a symbol having an asymptotic expansion
\[
\sum_{k=0}^\infty a^{(k)}(t,x,\xi),
\] 
where each $a^{(k)}(t,\bullet)$ is homogeneous of degree $-k$ outside a compact set, and $\phi(t,\bullet)$ which is homogeneous of degree $1$. Recalling that $b = e^{i\phi_1} a$ and separating terms by homogeneity, first obtain from \eqref{doi-evolution} the eikonal equation
\[
\begin{cases} 
 \pa_t \phi_1 + \pbis_1(t,x,\xi)=0, \\ \phi_1(0,x,\xi)=0.
 \end{cases} 
\]
This equation is solved by \eqref{eq:phi1}, recalling that $\pbis_1(t)=p_1\circ \exp(t\hamvf_0).$ Next, obtain a sequence of transport equations, the first of which has the form
\[
\begin{cases} 
\pa_t a^{(0)} + \langle \xi, \nabla_x a_0\rangle = f(t,x,\xi)a^{(0)}\\
a^{(0)}(0,x,\eta) = 1,
\end{cases}
\]
where $f(t,x,\xi)$ is homogeneous of degree $0$.
Observe that this equation can be solved for all time since the characteristics are straight lines.
There are similar expressions for $a^{(k)}$ (with inhomogeneous term depending on $a^{(0)},\dotsc,a^{(k-1)}$ and with vanishing initial value). Let $\widetilde{a} \in \CI(\RR_t; \RR^{2d}\setminus \{0\})$ be such that 
\begin{equation} \label{eq:borelsum}
\widetilde a(t,x,\xi) \sim  \sum_{k=0}^\infty  a^{(k)}(t,x,\xi),
\end{equation}
and then set $a(t,x,\xi) = \zeta(x,\xi)\widetilde{a}(t,x,\xi)$, where $\zeta \in \CI(\RR^{2d})$ is such that $\zeta(x,\xi) = 0$ for $|(x,\xi)| \leq 1$ and $\zeta(x,\xi) = 1$ for $|(x,\xi)| \geq 2$. Thus $a$ is everywhere smooth, and $\phi_1$ is also smooth on the support of $a$. 

Let $\trd(t)$ be given by \eqref{eq:approximaterd}, and  $\rd_N(t)$ be the corresponding integral when \eqref{eq:borelsum} is summed from $0$ to $N$. There are two errors when applying $(i\partial_t -H)$ to $\rd_N(t)$: the first arises since the eikonal and transport equations are only satisfied outside a compact set, hence the corresponding error is residual. The second error arises since the corresponding amplitude $a_N$ is only a finite sum of terms. For this we simply cite \cite{HeRo81}*{Lemma III.6} for mapping properties of the corresponding oscillatory integral with amplitude $c^{(N+1)}(t,x,\xi)$. Since $N$ is arbitrary, the proof is complete.
\end{proof}

Observe that $\trd(t)(x,y)$ is indeed the distributional kernel of an operator $\schwartz' \rightarrow \schwartz'$ as described in \S \ref{subsect:oscillatory}: clearly the quadratic form $\left<x-y,\xi\right>$ satisfies the hypotheses \eqref{eq:nondegeneratephase}, and as in the proof of Lemma \ref{lem:parametrix-perturbed} is may be assumed that $\phi_1$ is smooth on the support of $a$.

\subsection{Composition}
In this section we analyze the composition $\tprop(t) = U_0(t) \trd(t) ,$ which will
give a parametrix for $U(t)$. Observe that $\tprop(t)$ is well defined as an operator between tempered distributions, for example.

Although some information about the composition can be gleaned from the general theory in
\cite{Helffer84}*{Chapter 2}, a more precise description of the resulting phase is needed here; for this reason the calculations that follow will be explicit. Write
\begin{align*}
\trd(t) &= \int e^{\ang{(x-y,\eta} + i\phi_1(t,x,\eta)} b_1(t,x,\eta)\, d\eta, \\
U_0(t) &= \int e^{i\phi_2(t,x,\eta)- i\left<y,\eta\right>} b_2(t,x,\eta) \,d\eta,
\end{align*}
for appropriate amplitudes $b_j \in\CI(\RR_t;  \symb^0)$, where $\phi_2$ is given by \eqref{eq:mehlerphase}, and $\phi_1$ is given by \eqref{eq:phi1}.  Of course the formula for $U_0(t)$ only makes sense if $t -2\pi n \in (-\pi/2,\pi/2)$ for some $n\in \ZZ$. As remarked at the end of the previous section, it may be assumed that $\phi_1$ is smooth on the support of $b_1$.

Formally then, the composition has Schwartz kernel
\[
\tprop(t)(x,y) = \int e^{i\phi_2(t,x,\eta)-i\left< y,\eta \right> + i\phi_1(t,y,\eta)} b(t,x,y,\eta) \,d\eta,
\]
where the amplitude $b = b(t,x,y,\eta)$ is given by
\begin{align} \label{eq:compositionamplitude}
b(t,x,y,\eta) &= \int e^{i\left<z-y,\xi-\eta\right> + i(\phi_1(t,z,\xi) - \phi_1(t,y,\eta))} b_2(t,x,\eta) b_1(t,z,\xi) \,dzd\xi \notag \\
&= \int e^{i\left<z,\xi\right>} e^{i\phi_1(t,y+z,\eta+\xi) - i\phi_1(t,y,\eta)} b_2(t,x,\eta) b_1(t,y+z,\eta+\xi) \, dzd\xi.
\end{align}
In analyzing the latter integral, there is no difficulty in supposing more generally that $b_j \in\CI(\RR_t;  \symb^{m_j})$ for some $m_j \in \RR$. Since all the dependence on $t$ henceforth will be smooth and parametric, for notational simplicity the dependence on $t$ will be suppressed. Define
\[
a_0(x,y,z,\eta,\xi) = e^{i\phi_1(y+z,\eta+\xi) - \phi_1(y,\eta)} b_2(x,\eta) b_1(y+z,\eta+\xi).
\]
While $b_j$ have improved decay under differentiation for $j=1,2$, this is not the case for $a_0$ due to the homogeneous of degree $1$ phase factor. Thus 
\[
|\partial^\alpha a_0| \leq C_{\alpha}\left<(x,\eta)\right>^{m_2} \left<(y+z,\eta+\xi) \right>^{m_1}
\]
for each $\alpha$. Now integrate by parts using the operator $L = (1+|z|^2 + |\xi|^2)^{-1}(1+\Lap_z + \Lap_\xi)$ to see that 
\begin{equation} \label{eq:integratewithL}
|(L^t)^k \,\pa^\alpha _{x,y,\eta} a_0| \leq C_{\alpha k} \left<(x,\eta)\right>^{m_2} \left<(y,\eta)\right>^{m_1} \left<(z,\xi)\right>^{|m_1|-2k}.
\end{equation}
Choosing $k > d+ |m_1|/2$ shows that $b$ given by \eqref{eq:compositionamplitude} is smooth and satisfies
\[
|\partial^{\alpha} b| \leq C_{\alpha} \left<(x,\eta)\right>^{m_2} \left<(y,\eta)\right>^{m_1}
\]
for each $\alpha$. 

This result must be improved to include symbol bounds when $x=y$;  this is important when taking the distributional trace of $\tprop(t)$. 

\begin{lemma} \label{lem:composition}
	The pullback of the amplitude $b$ by the map $(t,x,\eta) \mapsto (t,x,x,\eta)$ lies in $\CI(\RR_t; \symb^{m_1+m_2})$.
\end{lemma}

\begin{proof}
As in the previous paragraph the smooth dependence on $t$ will follow immediately by differentiating under the integral sign, and so to simplify notation the dependence on $t$ will be again be dropped. 

First, observe that it suffices to consider the integral \eqref{eq:compositionamplitude} over $|(z,\xi)| \leq  (1/2)|(x,\eta)|$, since on the complement $b(x,x,\eta)$ is rapidly decaying in $(x,\eta)$ by \eqref{eq:integratewithL}. So now define
\[
b_{\lambda}(x,\eta) = b(\lambda^{1/2}x,\lambda^{1/2}x,\lambda^{1/2}\eta),
\] 
where $1 \leq |(x,\eta)| \leq 2$. In order to prove the lemma it suffices to show the uniform bounds
\begin{equation} \label{eq:rescaledsymbolbounds}
\left|\pa^\alpha_{x,\eta} b_\lambda(x,\eta)\right|
\leq C_{\alpha} \lambda^{(m_1 + m_2)/2}
\end{equation}
as $\lambda \rightarrow \infty$. For this, define 
\[
g_{\lambda}(z,\xi, x,\eta) = a(\lambda^{1/2}x,\lambda^{1/2}\eta) b(\lambda^{1/2}(x+z),\lambda^{1/2}(\eta+\xi)),
\]
noting that
\begin{equation} \label{eq:boundonA}
|\partial^\alpha g_{\lambda}(z,\xi, x,\eta) | \leq C_{\alpha}\lambda^{(m_1
	+ m_2)/2}
\end{equation}
uniformly in $1 \leq |(x,\eta)| \leq 2$ and $|(z,\xi)| \leq 1/2$. A Taylor expansion of $\phi_1(z+x,\xi + \eta)$ at $(x,\eta)$ yields
\begin{align*}
\phi_1(z+x,\xi+\eta) &= \phi_1(x,\eta) + \left<z, \pa_x\phi_1(x,\eta)\right> + \left<\xi, \pa_\eta \phi_1(x,\eta)\right> \\ &+ \sum_{|\alpha| = 2} (z,\xi)^\alpha f_\alpha(x,z,\eta,\xi)
\end{align*}
for some smooth functions $f_\alpha$, so if we define
\[
\Phi_\mu(x,z,\eta,\xi) = z\xi + \mu\phi_1(z+x,\xi+\eta) - \mu\phi_1(x,\eta)
\]
for a parameter $\mu \in \RR$, then
\begin{align*}
\Phi_\mu = z\xi + \mu \big( z \pa_y\phi_1(y,\eta) + \xi \pa_\eta \phi_1(y,\eta) + \sum_{|\alpha| = 2} (z,\xi)^\alpha f_\alpha(z,\xi,x,\eta) \big).
\end{align*}
Using homogeneity of the phase, the rescaled amplitude $b_{\lambda}(x,\eta)$ can be written via a change of variables as
\begin{equation} \label{eq:amplitudeintegral} 
b_{\lambda}(x,\eta) = \lambda^d\int e^{i\lambda \Phi_{\mu}}  g_{\lambda}(z,\xi, x,\eta) \,dzd\xi
\end{equation}
by setting $\mu = \lambda^{-1/2}$. Let $C_\mu = \{ d_{z,\xi}\Phi_\mu = 0\}$ denote the set of
stationary points; thus $(z,\xi) \in C_\mu$ if and only if
\begin{align*}
\xi + \mu \pa_z \phi_1(z+x,\xi+\eta) &= 0,\\
z + \mu\pa_\xi \phi_1(z+x,\xi+\eta) &= 0.
\end{align*}
By the implicit function theorem, we can parametrize $(z,\xi)$ by $(\mu,x,\eta)$ near any fixed $(x_0,\eta_0)$ for $|\mu|$ sufficiently small, and obtain
\begin{align*}
|z(\mu,x,\eta)| + |\xi(\mu,x,\eta)| \leq C|\mu|.
\end{align*}
In particular these points satisfy $|(z,\xi)| \leq 1/2$ for $|\mu|$ sufficiently small and $1 \leq |(x,\eta)| \leq 2$, hence the derivative bounds \eqref{eq:boundonA} for
$g_{\lambda}$ will apply.

We can now estimate the integral \eqref{eq:amplitudeintegral} and its derivatives, initially treating $\mu$ as a parameter; assume without loss that $g_{\lambda}(z,\xi,x,\eta)$ vanishes for $|(z,\xi)| \geq 1/3$. Consider a typical derivative $\pa_{x,\eta}^\gamma g_{\lambda}$. This is a sum of terms,
where those with $\ell \leq \abs{\gamma}$ derivatives landing on the exponential factor can be
written as
\begin{equation}\label{derivsofamplitude}    
\begin{aligned}
\lambda^{d} (\lambda \mu)^{\ell} \int e^{i \lambda \Phi_{\mu}} \left(\pa^{\gamma'}_{x,\eta} g_{\lambda}\right) \sum_{\abs{\beta}=\ell} (z,\xi)^\beta h_\beta \, dz d\xi
\end{aligned}
\end{equation}
for some smooth functions $h_\beta = h_\beta(z,\xi, y,\eta,\mu)$ and $\abs{\gamma'}\leq \abs{\gamma}$. 

Now apply the method of stationary phase, recalling the bounds \eqref{eq:boundonA}.
At the critical set $C_\mu$, each term $(z,\xi)^\beta h_\beta(z,\xi,y,\eta)$ in \eqref{derivsofamplitude} gives an additional factor
of order $O(|\mu|^{\ell/2})$, since both critical points $z(\mu,y,\eta),\, \xi(\mu,y,\eta)$ are of
order $O(|\mu|)$. When $\mu = \lambda^{-1/2}$ this cancels with the factor
of $\lambda^{\ell/2}$ in front of the integral in
\eqref{derivsofamplitude}.  The stationary phase formula
eliminates the prefactor of $\lambda^d$, showing that
\[
|\partial^\alpha b_\lambda(x,\eta)| = O(\lambda^{(m_1 + m_2)/2})
\]
near $(x_0,\eta_0)$. Since the set where $1 \leq |(x,\eta)| \leq 2$ is compact, this implies the symbol estimates \eqref{eq:rescaledsymbolbounds} everywhere on the latter set.
\end{proof}

More generally, Lemma \ref{lem:composition} is true whenever $\phi_2$ is a quadratic form satisfying \eqref{eq:nondegeneratephase} and $\phi_1$ is homogeneous of degree $1$.

\begin{corollary} \label{cor:parametrix}
    If $t-2\pi n \in (-\pi/2,\pi/2)$ for some $n \in\ZZ$, then the Schwartz kernel of $\tprop(t)$ is given by an oscillatory integral
    \begin{align*}
        \tprop(t,x,y) = \int e^{i\phi_2(t,x,\eta) -i \left<y,\eta \right> + i\phi_1(t,x,\eta)} b(t,x,y,\eta)\, d\eta,
    \end{align*}
    where $\phi_2$ is given by \eqref{eq:mehlerphase}, and $\phi_1$ is given by \eqref{eq:phi1}.
    The pullback of $b \in \CI(\RR_t \times \RR^{3d})$ by the map $(t,x,\eta) \mapsto (t,x,x,\eta)$ lies in $\CI(\RR_t;\symb^0)$.
\end{corollary}

\begin{proof}
    This follows directly from Lemma \ref{lem:composition}.
\end{proof}

Let $R(t) = (i\partial_t - P(t))\trd(t)$. A brief calculation shows that $\tprop(t)$ satisfies the equation
\begin{equation}\label{eq:UNequation}
\left\{\begin{aligned}
(i\pa_t - H ) \tprop(t) &= U_0(t)R(t),\\
\tprop(0) &= \id + K.
\end{aligned}\right.
\end{equation}
It follows by Duhamel's principle that 
\begin{equation} \label{eq:duhamel}
\tprop (t) - U(t) = U(t) K - i \int_0^t U(t-s)U_0(t)R(t) \, ds.
\end{equation}
Recall that $U_0(t)$ and $U(t)$ both preserves the scale of isotropic Sobolev spaces. Since $R(t)$ is a smooth family of residual operators and $K$ is residual, it follows immediately from \eqref{eq:duhamel} that 
\begin{equation} \label{eq:parametrixaccuracy}
\widetilde{R}(t) = \tprop(t)- U(t) \in \CI(\RR_t;\mathcal{L}(\sob^{-N},\sob^N))
\end{equation}
for each $N$.

As in Lemma \ref{lem:parametrix-perturbed}, there is no loss in assuming that the amplitude $b(t,x,y,\eta)$ in $\tprop(t)$ is supported away from $(x,y,\eta)=0$: inserting a cutoff modifies $\tprop(t)$ by a residual operator which does not affect the error analysis above. In particular, it may be assumed that $\phi_1$ is smooth on the support of $b$.

\subsection{Propagation of classical singularities}
Let $u \in \mathcal{E}' + \schwartz$. We want to calculate the classical wavefront set $\WF(U(t)u)$ of $u$.

By Proposition~\ref{prop:isoWF}, if $\{ (0,\xi): \xi \in \RR^d\} \cap \WFiso(u) =
\emptyset$ then $u \in \CI.$
By Proposition \ref{prop:prop}, $u(t,\bullet) \in \CI$ except at times
$t \in \pi \ZZ,$ and at those times, we also know that
\[\WF(U_0(k\pi)u) = \{(-1)^k(x,\xi): (x,\xi) \in \WF(u)\}.\]
It remains to calculate how singularities are moved by the reduced propagator $F(t)$.

We now assume more generally that $u \in \schwartz'$.
Equation \eqref{eq:parametrixaccuracy} (and preceding discussion) implies that the parametrix constructed in Lemma \ref{lem:parametrix-perturbed} satisfies
$\rd - \trd \in \CI(\RR_t, \mathcal{L}(\schwartz', \schwartz))$.
The classical wavefront set is thus completely determined by the parametrix:
\begin{align*}
    \WF(\rd(t)u) = \WF( \trd(t) u + (\rd(t) - \trd(t))u) = \WF(\trd(t)u),
\end{align*}
because $(\rd(t) - \trd(t))u \in \schwartz \subset \CI$.

Recall that
\begin{align*}
    \trd(t) &= \int e^{i(\ang{x-y,\xi} + \phi_1(t,x,\xi))} a(t,x,\xi) d\xi\\
    &= \int e^{i\phi(t,x,y,\xi)} \tilde{a}(t,x,\xi) d\xi,
\end{align*}
with $\phi \equiv \ang{x-y,\xi} + \phi_1(t,0,\xi)$ and
\[\tilde{a}(t,x,\xi) \equiv e^{i(\phi_1(t,x,\xi) - \phi_1(t,0,\xi))} a(t,x,\xi).\]
Note that $\phi$ is homogeneous of degree one in $\xi$ and since, locally, $\phi_1(t,x,\xi) - \phi_1(t,0,\xi) \in S^0$ we see that
the amplitude is (locally) a Kohn-Nirenberg 1-symbol, $\tilde{a} \in S^0$.
Thus, the oscillatory integral $\trd(t)$ satisfies the assumptions of
Theorem 8.1.9 from \cite{Hormander1}, and we obtain the following:
\begin{proposition}\label{prop:WFF}
    The wavefront set of the integral kernel of $F(t)$ is given by
    \begin{align*}
        \WF(F(t)) \subset \left\{ (x,x-\pa_\xi \phi_1(t,0,\xi),\xi,-\xi) : x\in\RR^d, \xi \in \RR^d\setminus \{0\} \right\}.
    \end{align*}
\end{proposition}

If we want to calculate the wavefront set of $\rd(t)u$ for $u \in \schwartz'$ we have to show that there are no contributions to wavefront set coming from infinity.
Fix $t_0 \in \RR$ and let $K \subset \RR^d$ compact with $\chi_1 \in
\CcI(K);$ set
\[r = \max_{(x,\xi) \in K \times \RR^d} |x + \pa_\xi \phi_1(t_0,x,\xi)|.\]
Note that $\pa_\xi \phi_1$ is homogeneous of degree zero in $(x,\xi)$ and therefore $r < \infty$.
Let $\chi_2 \in \CI(\RR^d)$ with $\supp \chi_2 \cup B_{r+1}(0) = \emptyset$ and homogeneous of degree zero outside of $B_{r+2}(0)$.

It suffices to show that $\chi_1(x)\chi_2(y) \trd(t_0,x,y) \in \schwartz(\RR^{2d})$.
Set $\Phi = \ang{x-y,\xi} + \phi_1(t_0,x,\xi)$ and define the operator $L$ by
\begin{align*}
    Lu = \frac{\ang{\pa_\xi \Phi, D_\xi u}}{\abs{\pa_\xi\Phi}^2}.
\end{align*}
$L$ is well-defined on $\supp \chi_1(x)\chi_2(y)$ and satisfies $L e^{i\Phi} = e^{i\Phi}$ and
for all $a \in \symb^m$ and $N \in \NN$,
\begin{align*}
    \abs{(L^t)^N a(x,\xi)} &\leq C \ang{x-y + \pa_\xi\phi_1(t_0)}^{-N}\ang{(x,\xi)}^{m-N}\\
    &\leq C \ang{y}^{-N} \ang{\xi}^{m-N}.
\end{align*}
Integration by parts with this operator shows that $\chi_1(x)\chi_2(y)
\trd(t_0,x,y)$ and all its derivatives are rapidly decaying, hence for
any $u \in \schwartz',$ we know that $\WF F(t) u\cap \pi^{-1} K$ is determined by
the restriction of $u$ to $B_{r+1}(0),$ and is as follows:
\begin{proposition}
   For $u \in \schwartz',$
    \begin{align*}
        \WF(F(t)u) = \left\{(x-\pa_\xi \phi_1(t,0,\xi),\xi)\colon (x,\xi) \in \WF(u)\right\}.
    \end{align*}
\end{proposition}
\begin{proof}
The usual calculus of wavefront sets, together with
Proposition~\ref{prop:WFF}, shows that
\begin{equation}\label{WFcontain}    \begin{aligned}
        \WF(F(t)u) \subset \left\{ (x-\pa_\xi \phi_1(t,0,\xi),\xi)\colon (x,\xi) \in \WF(u)\right\}.
    \end{aligned}
\end{equation}
It remains to upgrade this containment of sets to equality.  To do this, we
simply observe that by the calculus of wavefront sets and a second
use of Proposition~\ref{prop:WFF},
$$
\WF F^*(t)u \subset  \left\{ (x,\xi)\colon (x-\pa_\xi \phi_1(t,0,\xi),\xi) \in \WF(u)\right\}.
$$
On the other hand $F(t)^* F(t) =I,$ hence the containment in
\eqref{WFcontain} must have
been equality.
\end{proof}

\begin{corollary}
    Let $u \in \schwartz'$ and $k \in \ZZ$. The wavefront set of the full propagator is given by
    \begin{align*}
        \WF(U(\pi k)u) = \left\{(-1)^k\left(x+\int_0^{\pi k} \pa_\xi (\pbis_1(t,0,\xi)) dt,\xi\right)\colon (x,\xi) \in \WF(u)\right\}.
    \end{align*}
    If $t \not \in \pi \ZZ$ and $u \in \mathcal{E}' + \schwartz$ then $\WF(U(t)u) = \emptyset$.
\end{corollary}
For $t = 2\pi k$ this becomes \[\WF(U(2\pi k)u) = \{ (x + k\pa_\xi
(\Xray p_1)(0,\xi), \xi)\colon (x,\xi) \in \WF(u)\}.\]

\subsection{Traces} \label{subsect:traces} Recall that $\Tr U(t)$ is well defined as a tempered distribution. More precisely, if $\chi \in \schwartz(\RR)$, then the Schwartz kernel of
\begin{equation} \label{eq:tracekernel}
\int \chi(t) U(t) \, dt
\end{equation}
lies in $\schwartz(\RR^{2d})$, hence the operator is of trace-class. Indeed, if $\{e_j\}$ is an orthonormal basis for $L^2(\RR^d)$ consisting of eigenvectors of $H$ with corresponding eigenvalues $\lambda_j$, then \eqref{eq:tracekernel} has Schwartz kernel
\[
\sum_{j = 0}^\infty \hat \chi(\lambda_j) e_j(x) e_j(y),
\]
which converges in $\schwartz(\RR^{2d})$ since $\hat{\chi}$ is rapidly decreasing. In order to obtain results on singularities of $\Tr U(t),$ it suffices to study the
trace of $\tprop(t)$ and its Fourier transform (cf.\ Lemme (IV.1) of \cite{HeRo81}):

\begin{lemma}\label{lem:smooth-trace}
If $\chi \in \CI_c(\RR)$, then $\widetilde{R}(t) = \tprop(t) - U(t)$ is of trace class, and
\begin{align*}
    \big| \Tr \int e^{it\lambda}\chi(t)\widetilde{R}(t) \,dt\,\big| \leq  C_k \left<\lambda \right>^{-N}
\end{align*}
for each $\lambda \in \RR$ and $N>0$.
\end{lemma}

\begin{proof}
For $N\gg 0$ operators in $\mathcal{L}(\sob^{-N}, \sob^N)$ are of
trace-class (see \cite{Hormander3}*{Lemma 19.3.2}). Using repeated integration by parts, the claim follows from \eqref{eq:parametrixaccuracy}.
\end{proof}

On the other hand, if $\chi \in \CI_c((-\pi/2,\pi/2))$, then the operator
\[
\int \chi(t-2\pi n) \tprop(t) \, dt
\]
also has its Schwartz kernel in $\schwartz(\RR^{2d})$. Replacing $\chi$ with $e^{it\lambda}\chi$, it follows that the trace of $\F_{t \to \lambda}^{-1}  \chi(t-2\pi n) \tprop(t)$ is
\[
 (2\pi)^{-1}  \int e^{it\lambda + i\phi_2(t,x,\eta) -i \left<x,\eta \right> + i\phi_1(t,x,\eta)} \chi(t-2\pi n) b(t,x,x,\eta)\, dt dx d\eta.
\]
In the next section we will evaluate this integral as $\lambda \rightarrow \infty$.

\section{Stationary phase}\label{sec:nonhomog-phase}

In this section we apply the method of stationary phase to evaluate an integral of the form
\begin{equation} \label{eq:mainintegral}
I(\lambda) = \int e^{i(t\lambda+\psi_2(t,x,\eta)+\psi_1(t,x,\eta))} \chi(t) a(t,x,\eta) \, dt dx d\eta
\end{equation}
as $\lambda \rightarrow \infty$, where $\chi \in \CI_c(\RR)$. Letting $(r,\theta)$ denote polar coordinates on $\RR^{2d}$, we will also express various functions of $(x,\eta)$ in terms of $(r,\theta)$. The assumptions are as follows:
\begin{enumerate} \itemsep6pt
	\item $\psi_j(t,\bullet)$ is homogeneous of degree $j$,
	\item $a \in \CI(\RR_t; \symb^0(\RR_x^d))$, and $\psi_j$ are smooth on the support of $a$,
	\item there exists a unique $t_0 \in \supp \chi$ such that $\psi_2(t_0,\bullet)=0$,
	\item there exists a unique $r_0 > 0$ such that $\partial_t \psi_2(t_0,r_0,\theta) = -1$ for all $\theta \in \Sph^{2d-1}$.
\end{enumerate}
\noindent  Define the set where the restriction of $\nabla\psi_1(t_0,\bullet)$ to ${\Sph^{2d-1}}$ vanishes to infinite order,
\begin{align*}
    \Pi_{t_0} = \left\{ \theta \in \Sph^{2d-1}:   \pa_\theta^\alpha (\psi_1(t_0,1,\theta)) = 0 \text{ for all } \alpha \in \NN^{2d-1}\setminus 0 \right\}.
\end{align*}
We can now state our main result on the asymptotics of $I(\lambda)$:

\begin{proposition}\label{prop:osc-int}
	If $\Pi_{t_0}$ has measure zero, then the integral \eqref{eq:mainintegral} satisfies
	\begin{align*}
	I(\lambda) = o(\lambda^{d-1}).
	\end{align*}
If, instead, the restriction of $\psi_1(t_0,\bullet)$ to $\Sph^{2d-1}$ is Morse-Bott with $k>0$ non-degenerate
	directions, then
	\begin{align*}
	I(\lambda) = O(\lambda^{d-1-k/4}).
	\end{align*}
\end{proposition}

\begin{proof}
	To begin, rewrite the integral \eqref{eq:mainintegral} in polar coordinates, and then make the change of variables $r \mapsto \lambda^{1/2}r$. By homogeneity of the phases,
	\begin{equation} \label{eq:Ipolar}
	I(\lambda) = \lambda^d \int e^{i \lambda(\psi_2(t,r,\theta) + \lambda^{-1/2} \psi_1(t,r,\theta) +t)} \chi(t) a(t,\lambda^{1/2} r,\theta) \, dt\,  r^{2d-1} dr  d\theta.
	\end{equation}
	Observe that the exponential term in this integral can be written as  $\exp(i\lambda\Psi_{\mu})$, where
	\[
	\Psi_\mu(t,r,\theta) = \psi_2(t,r,\theta) + \mu \psi_1(t,r,\theta) + t
	\]
	and $\mu = \lambda^{-1/2}$.
	The proof proceeds in two steps.
	
	\step{Stationary phase in $(t,r)$} First we apply the method of
	stationary phase to the variables $(r,t)$ for $|\mu|$ small, treating $\mu$ and $\theta$ as parameters. Let 
	\[
	C_{\mu} = \{(t,r)\colon d_{r,t}\Psi_\mu(t,r,\theta) = 0\}
	\]
	denote the corresponding stationary set. Now $(r\pa_r)\psi_j=j\psi_j$ by homogeneity of the phases, so the stationary points are where
	\begin{equation} \label{eq:stationarypoints} \begin{cases} 
	2 \psi_2 + \mu \psi_1=0,\\
	\pa_t \psi_2 + \mu \pa_t\psi_1+1=0. \end{cases}
	\end{equation} 
	By hypothesis, if $\theta_0 \in \Sph^{2d-1}$ is fixed and $\mu=0$, then these equations are satisfied on the support of the function $(t,r)\mapsto \chi(t) a(t,\lambda^{1/2}r,\theta_0)$ precisely when $t=t_0,\ r=r_0$. 
	
Using the implicit function theorem,  parametrize $C_{\mu} \cap \supp(\chi \cdot a)$ near $\theta_0$ for small $|\mu|$. Indeed, differentiating the equations \eqref{eq:stationarypoints} in $(t,r)$ at $\mu=0,\, r=r_0, \,t=t_0$  yields the invertible Hessian matrix
	$$
	\begin{pmatrix}
	0 & -2 \\ -2 & \pa_t^2\psi_2
	\end{pmatrix}.
	$$
Denote by $t= t(\mu,\theta)$ and $r = r(\mu,\theta)$ the corresponding critical points. Furthermore, by the implicit function theorem
		$$
	\begin{aligned}
	\begin{pmatrix}
	\pa_\mu r\\ \pa_\mu t
	\end{pmatrix}
=
\frac 14 \begin{pmatrix}
\psi_1 \pa_t^2 \psi_2 + 2 \pa_t \psi_1\\
2 \psi_1
\end{pmatrix}
\end{aligned}
	$$
at $\mu = 0, \, r=r_0, \, t=t_0$. Now Taylor expand $\Psi_\mu(t(\mu,\theta),r(\mu,\theta),\theta)$ at $\mu = 0$ to find that
\[
\Psi_\mu(t(\mu,\theta),r(\mu,\theta),\theta) =t_0+\mu \psi_1(t_0,r_0,\theta)+\mu^2 \gamma(\mu,\theta)
\]
near $\mu =0, \, \theta = \theta_0$, where $\gamma=\gamma(\mu,\theta)$ is a smooth function of $\mu$ and $\theta.$ 

Next, apply the method of stationary phase to the integral
\[
J(\lambda,\mu,\theta) = \lambda^d \int e^{i\lambda\Psi_\mu }\chi(t) a(t,\lambda^{1/2} r,\theta) \, dt\,  r^{2d-1} dr,
\]
treating $\theta \in \Sph^{2d-1}$ and $\mu$ as parameters. In fact, it may be assumed $a(t,\lambda^{1/2}r,\theta)$ has support on $\{r \leq 3r_0\}$. Indeed, consider the following operator, which is well defined on $\{r \geq 2r_0\} \cap \supp \chi$:
\[
L = \lambda^{-1}((\partial_t \phi_2 + 1)^2 + 4\phi_2^2)^{-1} \left(  (\partial_t \phi_2 + 1)\partial_t + 2\phi_2 \partial_r \right).
\]
Due to the symbol bounds on $a$,
\[
|(L^t)^k ( e^{i \lambda^{1/2} \psi_1} \chi(t) a(t,\lambda^{1/2}r,\theta) r^{2d-1} )| \leq C_k \lambda^{-k/2} r^{2d-1-2k}.
\]
Inserting a cutoff to $\{r \geq 2r_0\}$ in the integrand of \eqref{eq:Ipolar} and integrating by parts using $L$ gives a contribution of order $O(\lambda^{-\infty})$. By stationary phase, for any $M \geq 1$,
\[
J(\lambda,\mu,\theta) = \lambda^{d-1} e^{i\lambda (t_0 + i\mu \psi_1(t_0,r_0,\theta))} a_M(\lambda^{1/2},\mu,\theta) + O(\lambda^{d-1-M})
\]
uniformly in $\theta$ for $|\mu|$ sufficiently small; here, $a_M$ is a function depending smoothly on $(\lambda^{1/2},\mu,\theta)$. Note that while successive terms in the stationary
phase expansion involve differentiation of $a(t,\lambda^{1/2} r ,\theta)$
with respect to $r,$ the symbol estimates on $a$ ensure uniform bounds on each $a_M$ as $\lambda\rightarrow \infty$. 

\step{Stationary phase in $\theta$}
Recall that $I(\lambda)$ is the integral of $J(\lambda,\lambda^{-1/2},\theta)$ over $\Sph^{2d-1}$ with respect to $\theta$. In other words, for each $M$,
\begin{align}\label{eq:reduced-int}
I(\lambda) = \lambda^{d-1}e^{i\lambda t_0} \int
e^{i\lambda^{1/2}\psi_1(t_0,r_0,\theta)}
a_M(\lambda^{1/2},\lambda^{-1/2},\theta) \, d\theta + O(\lambda^{d-1-M}).
\end{align}
We now complete the proof of Proposition \ref{prop:osc-int}. If $(x,\eta) \in \Sph^{2d-1} \setminus \Pi_{t_0}$, then there exists an $\alpha \in \NN^{2d-1}$ such that
\begin{align*}
\pa_\theta^\alpha \psi_1 \neq 0,
\end{align*}
in a neighborhood of $(x,\eta)$ within $\Sph^{2d-1}$. By the weak stationary phase lemma for degenerate stationary points \cite{Stein93}*{p.~342, Proposition 5} and a covering argument, the contribution of the integral over $\Sph^{2d-1} \setminus \Pi_{t_0}$ is $o(\lambda^{d-1})$ (cf.~\cite{GuSa88}).
Therefore,
\begin{align*}
I(\lambda) = \lambda^{d-1}e^{i\lambda t_0} \int_{\Pi_{t_0}} e^{i\lambda^{1/2}\psi_1(t_0,r_0,\theta)}
b(\lambda^{1/2},\lambda^{-1/2},\theta) \, d\theta+ o(\lambda^{d-1}).
\end{align*}
This implies that if $\Pi_{t_0}$ is of measure zero then
\begin{align*}
I(\lambda) = o(\lambda^{d-1}),
\end{align*}
which proves the first part of Proposition \ref{prop:osc-int}. For the second part, the condition that $\psi_1(t_0,r_0,\bullet)$ is Morse-Bott with $k$ nondegenerate directions implies that $I(\lambda) = O(\lambda^{d-1-k/4})$ by \cite{Hormander1}*{Theorem 7.7.6}, so taking $M \geq k/4$ finishes the proof.
%
%
\end{proof}

\section{Spectral asymptotics}

\subsection{Singularity at $t=0$} \label{sec:sing-zero}
In this section we calculate the leading order asymptotics of the singularity of $\Tr U(t)$ at $t = 0$. More precisely, we obtain the $\lambda \rightarrow \infty$ behavior of its inverse Fourier transform, after a suitable mollification. For this we use a short-time parametrix for $U(t)$ constructed in \cite{HeRo81}. This construction actually applies to any self-adjoint classical elliptic isotropic operator of order $2$, and for this reason we state Proposition \ref{prop:leading-asymptotic} below quite generally. 

Let $p \in \symbcl^2(\RR^d)$ be real-valued and elliptic, and then set $P = \Op_W(p)$. Denote by $N(\lambda) = \sum_{\lambda_j \leq \lambda} 1$ the counting function for the eigenvalues of $P$.

\begin{proposition}\label{prop:leading-asymptotic}
     Let $\rho \in \schwartz(\RR)$ be such that $\hat \rho$ has compact support in $(-\ep,\ep)$. If $\ep  > 0$ is sufficiently small, then
    \begin{align*}
        (N * \rho)(\lambda) &=
        (2\pi)^{-d}\int_{\{p_2 + p_1 \leq \lambda\} } dxd\eta - (2\pi)^{-d}\int_{\{p_2 = \lambda\} } p_0(x,\eta) \frac{dS}{|\nabla p_2|}\\
        &+ O(\lambda^{d-3/2}).
    \end{align*}
\end{proposition}

\begin{proof}
	Let $U(t)$ denote the Schr\"odinger propagator for $P$. As remarked above, we will use a parametrix $U_N(t)$ for $U(t)$ taken from \cite{HeRo81}, which exists on some time interval $(-\ep,\ep)$ (note that $U_N(t)$ differs from the long time parametrix constructed in Corollary \ref{cor:parametrix}). In the notation of \cite{HeRo81},
	\[U_N(t,x,y) = (2\pi)^{-d} \int e^{i(S_2(t,x,\eta) -\left<y,\eta\right>+ S_1(t,x,\eta))} a_N(t,x,\eta)\, d\eta.
	\]
	Here $S_2,S_1$ are appropriate phase functions, and the symbol $a_N$ is a finite sum 
	\[
	a_N(t,x,\eta) = \sum_{k=0}^N a^{(k)}(t,x,\eta),
	\]
	where each $a^{(k)}(t,\bullet)$ is homogeneous of degree $-k$ outside a compact set and vanishes near $(x,\eta) =0$. Note, however, that in \cite{HeRo81} the operator $P$ is the left quantization of $p$ rather than its Weyl quantization. In order to extract the leading order behavior of these quantities, first write 
	\[
	\Op_W(p) = \Op_L(\widetilde{p})
	\] with $\widetilde{p} \in \symbcl^2$ and $\widetilde{p}_j = p_j$ for $j = 1,2$, but
    \begin{equation}\label{eq:change-quantization}
        \widetilde{p}_0 = p_0 - (i/2) \left<\pa_x, \pa_\xi\right> p_2.
    \end{equation}
    Referring to \cite{HeRo81}*{Equations 37-38} for the transport equations satisfied by $a^{(k)}$ and using \eqref{eq:change-quantization}, we find that
	\begin{equation*} 
	 a_N(0,x,\eta) = 1, \quad \pa_t a^{(0)}(0,x,\eta) = -ip_0 -(1/2)\left<\partial_x, \partial_\xi\right> p_2.
	\end{equation*}
    Recalling that $\F_{\lambda \to t} N'(\lambda) = \Tr U(t)$, we have $\F_{\lambda \to t} \{ N' * \rho \} = \hat\rho(t) \Tr U(t)$. Motivated by this, define the distribution $K(t) = \hat{\rho}(t) \Tr U_N(t)$, so that
	\begin{align*}
	K(t) = (2\pi)^{-d}\hat\rho(t) \int e^{i(S_2(t,x,\eta) - \left<x,\eta\right> + S_1(t,x,\eta))} a_N(t,x,\eta) \, dx d\eta.
	\end{align*}
	This makes sense so long as $\hat{\rho}(t)$ has support on the interval where $U_N(t)$ is well defined.
	
    By \cite{HeRo81}*{Equations 35-36}, $S_2(0,x,\eta) = \left< x,\eta \right>$ and $S_1(0,x,\eta) = 0$, so by Taylor's theorem
	\[
	 S_2(t,x,\eta) - \left<x,\eta\right> + S_1(t,x,\eta) = t \psi(t,x,\eta)
	\]
    with $\psi$ a smooth function. More precisely, $\psi$ is given to leading order in $t$ by
	\begin{align*}
	\psi(t,\bullet) &= -({p}_2 + {p}_1) + (t/2) (\left< \pa_\xi {p}_2, \pa_x {p}_2\right> + \left<\pa_\xi p_1, \pa_x {p}_2\right> + \left< \pa_\xi {p}_2, \pa_x {p}_1 \right>)  +t^2 r(t,\bullet).
	\end{align*}
    We now follow the argument of \cite{Hormander4}*{Lemma 29.1.3}. First, define
	\begin{align*}
	A(t,\lambda) = (2\pi)^{-d}\int_{\{-\psi(t) \leq \lambda\} } a_N(t,x,\eta) \hat \rho(t) \,dx d\eta.
	\end{align*}
	Now for sufficiently small $|t|$, the function $-\psi(t,\bullet)$ is elliptic in $\symbcl^2$, and as in the aforementioned lemma 
	\[
	A(t,\lambda) \in S^d(\RR_t; \RR_\lambda)
	\] 
	is a Kohn--Nirenberg symbol for $|t|$ sufficiently small (see \eqref{eq:KN}).  Furthermore, it is an exercise in distribution theory to see that
	\begin{align*}
	K(t) = \int_\RR e^{-it\lambda} \pa_\lambda A(t,\lambda) \,d\lambda.
	\end{align*}
    Thus $K(t)$ is a conormal distribution, which can be written as the Fourier transform of a symbol by applying \cite{Hormander3}*{Lemma 18.2.1}. If we let $B(\lambda) = e^{iD_t D_\lambda} A(t,\lambda)|_{t=0}$ and recall the definition of $K(t)$, then
	\begin{align*}
	\F_{t \to \lambda}^{-1}\{\hat\rho(t) \Tr U_N(t)\}(\lambda) = \pa_\lambda B(\lambda).
	\end{align*}
	Expand $B(\lambda) = A(0,\lambda) - i\partial_t \partial_\lambda A(0,\lambda) + R(\lambda)$, 
	where $R \in S^{d-2}(\RR).$
	Also let $dS$ denote the induced surface measure on $\{p_2 = \lambda\}$. First, 
	\begin{align*}
	A(0,\lambda) = (2\pi)^{-d}\int_{ \{p_2 + p_1 \leq \lambda\} } dxd\eta.
	\end{align*}
	For the next term in the expansion, recall that
        $a_N(0,x,\eta)=1$ and compute
    \begin{align*}
        -i \pa_t A(0,\lambda) &= (2\pi)^{-d} \langle -i\pa_t a, H(\psi + \lambda) \rangle \big\rvert_{t=0} - i(2\pi)^{-d} \langle a\pa_t \psi, \delta(\psi + \lambda)\rangle\big\rvert_{t=0}\\
        &= -(2\pi)^{-d} \langle \tilde{p}_0, H(\lambda - p_2)\rangle - (i/2) (2\pi)^{-d} \langle \left<\pa_x p_2, \pa_\xi p_2\right>, \delta(\lambda - p_2) \rangle +e(\lambda)
    \end{align*}
    for some $e(\lambda) \in S^{d-1/2}(\RR).$ Here $H$ denotes the Heaviside function, and the pairings are in the sense of distributions. Integration by parts furthermore yields
    $$
        \langle  \left<\pa_x p_2, \pa_\xi p_2 \right>, \delta(\lambda-p_2) \rangle = \langle \left<\pa_x, \pa_\xi\right> p_2, H(\lambda-p_2) \rangle.
$$
    Since the pullback of $\delta$ is given by
    $\delta(\lambda - p_2) = |\nabla p_2|^{-1}dS,$  compute from
    \eqref{eq:change-quantization} that
    \begin{align*}
        -i\pa_\lambda\pa_t A(0,\lambda) &= -(2\pi)^{-d} \langle \tilde{p}_0 + (i/2)\left<\pa_x,\pa_\xi\right> p_2, \delta(\lambda - p_2)\rangle + O(\lambda^{d-3/2})\\
        &= -(2\pi)^{-d} \int_{ \{p_2 = \lambda\} } p_0 |\nabla p_2|^{-1} dS + O(\lambda^{d-3/2}).
    \end{align*}
	Finally, for any $k$,
    \begin{align*}
		(N' * \rho)(\lambda) &= \F^{-1}_{t\to\lambda } \{ \hat\rho \Tr U \}(\lambda) \\
		&= \F^{-1}_{t\to \lambda }\{\hat\rho \Tr U_N\}(\lambda) + O(\lambda^{-k}) \\
		&= \pa_{\lambda} B(\lambda) + O(\lambda^{-k}).
    \end{align*}
    provided $N = N(k)$ is sufficiently large (cf.\ Lemma IV.1 in \cite{HeRo81}).
    Integrating this equation gives the desired result.
\end{proof}

\subsection{Proof of Theorem \ref{maintheorem}}
We now return to the setting of Theorem \ref{maintheorem}, so that in Proposition \ref{prop:leading-asymptotic} we take the operator $P = H$. 
Begin by fixing an appropriate cutoff function in the time domain. Choose a real valued function $\rho \in \schwartz(\RR)$ with the following properties:
\begin{enumerate} \itemsep6pt 
	\item $\rho(\lambda) > 0$ for all $\lambda \in \RR$,
	\item $\hat\rho(t) = 1$ on $(-\ep,\ep)$ for some $\ep \in (0,\pi/2)$,
	\item $\supp \hat\rho \subset (-\pi/2,\pi/2)$, 
	\item $\rho$ is even.
\end{enumerate}
In order to compare $N(\lambda)$ with $(N * \rho)(\lambda)$, we will need
the following Fourier Tauberian theorem, from the appendix of
\cite{SaVa97}.  This result is implicit in \cite{DuGu75}, and has its
roots in \cites{Levitan52,Hormander68}.

\begin{lemma}[Theorem B.5.1 in \cite{SaVa97}]\label{thm:tauberian}
    Let $\rho$ be as above, and $\nu \in \RR$. If $(N' * \rho)(\lambda) = O(\lambda^\nu)$ and
    \[
        (N' * \chi)(\lambda) = o(\lambda^\nu)
    \]
    for each function $\chi$ satisfying $\hat\chi \in \CcI(\RR)$, $\supp \hat\chi \subset (0, +\infty)$, then
    \begin{align*}
        N(\lambda) = (N * \rho)(\lambda) + o(\lambda^\nu).
    \end{align*}
\end{lemma}

In order to prove Theorem \ref{maintheorem} it suffices to establish
\eqref{eq:infiniteorder} and \eqref{eq:morsebott}, since then the Weyl
law \eqref{eq:twotermweyl} is an immediate corollary of Lemma
\ref{thm:tauberian}. Indeed, using Proposition \ref{prop:wavetrace}
and a suitable partition of unity, either of the conclusions
\eqref{eq:infiniteorder} or \eqref{eq:morsebott} implies that
\begin{align*}
\F_{t \to \lambda}^{-1} \{\chi(t)\Tr U(t)\}(\lambda) = o(\lambda^{d-1}).
\end{align*}
for any function $\chi \in \CI_c(\RR)$ with $\supp \chi \subset  (0,\infty)$ (here $\chi$ is playing the role of $\hat \chi$ in Lemma \ref{thm:tauberian}). Now Proposition \ref{prop:leading-asymptotic} in particular shows that
\[
(N * \rho)(\lambda) = O(\lambda^d),
\]
which together verify the hypotheses of Lemma \ref{thm:tauberian}. This establishes the two term asymptotics \eqref{eq:twotermweyl} for $N(\lambda)$.

Thus, we aim to show \begin{align*}
\F_{t \to \lambda}^{-1} \{\chi(t)\Tr U(t)\}(\lambda) = o(\lambda^{d-1})
\end{align*}
 whenever $\supp \chi \subset (2\pi n - \ep, 2\pi n + \ep)$, where $n \in \NN \setminus 0$ and $\ep \in (0,\pi/2)$. By Lemma \ref{lem:smooth-trace}, for any $N > 0$
 \begin{align*}
 \F_{t \to \lambda}^{-1}\{\chi(t)\Tr U(t)\}(\lambda) &= \F_{t \to \lambda}^{-1}\{\chi(t)\Tr \tprop(t)\}(\lambda) + O(\lambda^{-N}).
 \end{align*}
Now use Corollary \ref{cor:parametrix} to see that
 \begin{align*}
\F_{t \to \lambda}^{-1}\{\chi(t)\Tr \tprop(t)\}(\lambda) = \int e^{it\lambda}e^{i(\phi_2(t,x,\eta) - \left< x,\eta\right> + \phi_1(t,x,\eta))} \chi(t) a(t,x,\eta) \,dtdxd\eta.
 \end{align*}
Apply Proposition \ref{prop:osc-int} with
\[
\psi_2(t,x,\eta) = \phi_2(t,x,\eta) - \left<x,\eta\right>, \quad \psi_1(t,x,\eta) = \phi_1(t,x,\eta).
\]
Since $\phi_2(t,x,\eta) = \sec(t) (x\eta - \sin(t)(|x|^2 + |\eta|^2)/2)$ and $\chi$ is supported close to $2\pi n$, the hypotheses of Proposition \ref{prop:osc-int} for the phases $\psi_2,\psi_1$ and symbol $a$ are satisfied. Indeed, in the notation of the latter proposition, we take 
\[
t_0 = 2\pi n, \quad r_0 = \sqrt{2}.
\]
Now suppose that the restriction of $\nabla \Xray p_1$ to ${\Sph^{2d-1}}$
vanishes to infinite order only on a set of measure zero. Then
$\nabla \phi_1(2\pi n,\bullet) = -\nabla \Xray ^np_1 = -n \nabla \Xray p_1$,
so $\nabla \phi_1(2\pi n,\bullet)$ vanishes to infinite order
only on a set of measure zero in $\Sph^{2d-1}$ as soon as $n\neq 0$. In that case Proposition
\ref{prop:osc-int} shows that
\[
\F_{t \to \lambda}^{-1}\{\chi(t)\Tr \tprop(t)\}(\lambda) = o(\lambda^{d-1}).
\] 
Similarly, if the restriction of $\Xray p_1$ to $\Sph^{2d-1}$ is Morse--Bott with $k > 0$ nondegenerate directions, then $\phi_1(2\pi n,\bullet)$ has the same property for $n\neq 0$. This completes the proof of Theorem \ref{maintheorem}.\qed

\end{document}